\documentclass[12pt, twoside]{article}

\usepackage{amsmath,amssymb,amscd,amsfonts,amsthm,verbatim}
\usepackage[all,cmtip]{xy}
\usepackage[ansinew]{inputenc}   
\usepackage[english]{babel}
\usepackage{anysize}
\usepackage{doc}
\usepackage{exscale}
\usepackage{fontenc}
\usepackage{ifthen}
\usepackage{latexsym}
\usepackage{makeidx}
\usepackage{syntonly}
\usepackage{graphicx}
\usepackage{float}
\usepackage{array}
\usepackage[all,cmtip]{xy}

\usepackage[usenames,dvipsnames]{color}

\marginsize{3.1cm}{3.1cm}{3.1cm}{3.9cm}

\xyoption{all} \numberwithin{equation}{section}
\newtheorem{thm}{Theorem}[section]
\newtheorem{cor}[thm]{Corollary}
\newtheorem{lem}[thm]{Lemma}
\newtheorem{pro}[thm]{Proposition}
\newtheorem{df}[thm]{Definition}

\newtheorem{lem-df}[thm]{Lemma-Definition}
\newtheorem{conj}[thm]{Conjecture}

\theoremstyle{definition}
\newtheorem{rmk}[thm]{Remark}
\newtheorem{ex}[thm]{Example}
\newtheorem{rmk-df}[thm]{Remark-Definition}

\newtheorem*{thm*}{Theorem}

\newtheorem*{rmk*}{Remark}

\newcommand{\beq}{\begin{equation}}
\newcommand{\enq}{\end{equation}}
\newcommand{\beqn}{\begin{equation*}}
\newcommand{\enqn}{\end{equation*}}

\newcommand{\ra}{\rightarrow}

\newcommand{\hra}{\hookrightarrow}
\newcommand{\longra}{\longrightarrow}

\newcommand{\dra}{\dashrightarrow}

\newcommand{\mC}{\mathbb{C}}

\newcommand{\mH}{\mathbb{H}}

\newcommand{\mP}{\mathbb{P}}

\newcommand{\mU}{\mathbb{U}}

\newcommand{\caA}{\mathcal{A}}

\newcommand{\caC}{\mathcal{C}}
\newcommand{\caD}{\mathcal{D}}
\newcommand{\caE}{\mathcal{E}}
\newcommand{\caF}{\mathcal{F}}

\newcommand{\caH}{\mathcal{H}}

\newcommand{\caK}{\mathcal{K}}

\newcommand{\caM}{\mathcal{M}}

\newcommand{\caO}{\mathcal{O}}

\newcommand{\caU}{\mathcal{U}}

\newcommand{\caW}{\mathcal{W}}

\DeclareMathOperator{\Div}{div}
\DeclareMathOperator{\Hom}{Hom}

\DeclareMathOperator{\Ext}{Ext}
\DeclareMathOperator{\caExt}{{\mathcal Ext}}

\DeclareMathOperator{\Cliff}{Cliff}

\DeclareMathOperator{\rk}{rk}

\begin{document}

\title{On the rank of the flat unitary summand of the Hodge bundle}

\author{V\'ictor Gonz\'alez-Alonso, Lidia Stoppino, Sara Torelli\footnote{The first author was partially supported by ERC StG 279723 ``Arithmetic of algebraic surfaces'' (SURFARI) and the project MTM2015-69135-P of the spanish ``Ministerio de Econom\'ia y Competitividad''. The second author have been partially supported by FAR Uninsubria. The third author have been partially supported by Fondi dottorato Pavia. The second and third authors are members of G.N.S.A.G.A.--I.N.d.A.M and have been partially supported by MIUR (Italy) through  PRIN 2012 ``Spazi di Moduli e Teoria di Lie '', and PRIN 2015 ``Moduli spaces and Lie theory''}.}

\maketitle

\begin{abstract}
Let $f\colon S\to B$ be a non-isotrivial fibred surface. We prove that the genus $g$, the rank $u_f$  of the unitary summand of the Hodge bundle $f_*\omega_f$ and the Clifford index $c_f$ satisfy the inequality $u_f \leq g - c_f$. Moreover, we prove that if the general fibre is a plane curve of degree $\geq 5$ then the stronger bound 
$u_f \leq g - c_f-1$ holds. In particular, this provides a strengthening of the bounds of \cite{BGN} and of \cite{FNP}. The strongholds of our arguments are the deformation techniques developed by the first author in \cite{Rigid} and by the third author and Pirola in \cite{PT}, which display here naturally their  power and depht.
\end{abstract}


\section{Motivation and statement of the results}\label{sect-intro}

Let $f\colon S\to B$ be a fibred surface, i.e. a surjective morphism with connected fibres from a smooth projective surface $S$ to a smooth projective curve $B$. 
Call $g$ the arithmetic genus of the fibres, and suppose that $g\geq 2$. Let $q=h^0(S, \Omega_S^1)$ be the irregularity of the surface and $b=h^0(B, \omega_B)$ the genus of $B$.
Let $\omega_f = \omega_S \otimes \left(f^*\omega_B\right)^{\vee}$ be the relative dualizing sheaf. 
The {\em Hodge bundle} of $f$ is the rank $g$ vector bundle $E_f = f_*\omega_f$ on $B$. 
The recent papers \cite{CD:Answer,CD:Vector-bundles,CD:Semiample} have given new attention on the so-called Fujita decompositions of $E_f$ (see Theorem \ref{thm:Fuj-decomp})
\begin{equation} \label{eq:Fuj-decomp-surf}
E_f = \caO_B^{\oplus q_f} \oplus \caF = \caU \oplus \caA,
\end{equation}
where $q_f=q-b$ is the relative irregularity of $f$, $\caF$ is nef with $h^1\left(B,\caF\otimes\omega_B\right)=0$, $\caU$ is flat and unitary, and $\caA$ is ample.
We denote by $u_f$ the rank of the unitary summand $\caU$.
Notice that $ \caO_B^{\oplus q_f} \subseteq \caU$, hence $u_f\geq q_f$.

A natural issue arising in the study of  the geometry of $f$ is to find  numerical relations between $q_f$ and $u_f$ and the other invariants (such as  $g$, $b$, the self-intersection of the relative canonical divisor $K_f^2$, the relative Euler characteristic $\chi_f$, the relative topological characteristic $e_f= 12\chi_f-K_f^2$). 

In particular, the relations between the relative irregularity $q_f$ and the genus $g$ have been analyzed intensively since the first results and conjecture of Xiao in the '80's \cite{Xiao-P1}. 
Let us summarize the main known results. 

\begin{enumerate}
\item[(Q1)] (Beauville \cite[Appendix]{Deb-Beau}) $q_f\leq g$, and equality holds if and only if $f$ is trivial, i.e. $S$ is birational to $B\times F$ and $f$ corresponds to the first projection.
\item[(Q2)] (Serrano \cite{Ser-Iso}) If $f$ is isotrivial (i.e. its smooth fibres are all mutually isomorphic) but not trivial, then $q_f\leq \frac{g+1}{2}$.
\item[(Q3)] 
(Xiao \cite{Xiao-P1}, \cite{Xiao-5/6}) If $f$ is non-isotrivial and $B\cong \mP^1$, then $q_f=q\leq \frac{g+1}{2}$. For non-isotrivial fibrations with arbitrary $B$ the bound $q_f\leq \frac{5g+1}{6}$ holds.
\item[(Q4)] (Cai \cite{Cai}) If $f$ is non-isotrivial and the general fibre is either hyperelliptic or bielliptic, the same bound $q_f\leq \frac{g+1}{2}$ holds. 
\item[(Q5)] Pirola in \cite{Pir-Xiao} constructs for non-isotrivial fibrations a higher Abel-Jacobi map whose vanishing implies $q_f\leq \frac{g+1}{2}$. 
\item[(Q6)] (Barja-Gonz\'alez-Naranjo \cite{BGN}) If $f$ is non-isotrivial, then $q_f\leq g-c_f$, where $c_f$ is the Clifford index of the general fibre of $f$. 
Favale-Naranjo-Pirola \cite{FNP} have recently proven the stronger inequality $q_f\leq g-c_f-1$ for families of plane curves of degree $d \geq 5$.
\end{enumerate}
Xiao's original conjecture \cite{Xiao-P1} that the first bound of 
(Q3) holds for any non-trivial fibration has been disproved by Pirola in \cite{Pir-Xiao} (more counterexamples have been found later by Albano and Pirola in \cite{alb-pir}) and was modified in \cite{BGN} as follows.
\begin{conj}[Modified Xiao's conjecture for the relative irregularity]\label{conj:mxiao}
For any non-isotrivial fibred surface $f\colon S\to B$ of genus $g\geq 2$ it holds 
\begin{equation} \label{eq:BGN}
q_f\leq \left\lceil \frac{g+1}{2} \right\rceil.
\end{equation}
\end{conj}
The results of \cite{BGN} imply the conjecture in the (general) case of maximal Clifford index $c_f=\left\lfloor\frac{g-1}{2}\right\rfloor$. 
All counterexamples to the original conjecture found by Albano and Pirola satisfy equality for the modified one.

\medskip
In this paper we address the same question for the unitary rank $u_f$: what are the relations between $u_f$ and $g$?
Let us discuss the possible counterparts of the results described above.
\begin{enumerate}
\item[(U1)] Clearly $u_f\leq\rk E_f=g$ and equality holds if and only if $f$ is locally trivial (i.e. $f$ is a holomorphic fibre bundle): this follows immediately observing that $u_f=g$ if and only if $\chi_f=\deg E_f=0$.
\item[(U2)] In the case of isotrivial but non-trivial fibrations, there is no better bound for $u_f$, analgous to Serrano's bound for $q_f$. Indeed, it is easy to construct such fibred surfaces with $u_f = g$ as appropriate quotients of products of curves. We do not know wether isotrivial non-locally trivial fibrations satisfy some bound or $u_f$ can be arbitrarily close to $g$. 
\item[(U3)]  If  $B\cong \mP^1$, then clearly the flat unitary summand is trivial, so $u_f= q_f=q\leq \frac{g+1}{2}$ for the non-isotrivial case. Moreover, as observed in \cite[Lemma 3.3.1]{LZ}, Xiao's argument can be extended to bound the rank of the whole unitary summand, hence $u_f\leq \frac{5g+1}{6}$ for $f$ non-isotrivial (the authors only consider semistable fibrations, but Xiao's proof works identically for any fibration).
\item[(U4)] Lu and Zuo in \cite[Theorem 4.7]{LZ-hyper} prove that for a non-isotrivial hyperelliptic fibration there is a finite base change such that the flat summand of the new Hodge bundle is trivial, and moreover it is known that $u_f$ is non-decreasing by base change (Remark \ref{rmk: base change}), and so one can apply the bound for the relative irregularity proved by Cai and get the bound
$u_f\leq \frac{g+1}{2}$. 
\item[(U5)] Pirola's construction is specific to the case of the trivial part of $E_f$, it does not apply directly to the flat unitary summand.
\end{enumerate}

In case (4) the bound follows by trivializing the unitary flat part by base change. However, the examples of Catanese and Dettweiler precisely show that this is not always possible: indeed they prove that  the flat summand can be trivialized via a base change if and only if the image of the monodromy map associated to the unitary summand is finite (see also \cite{barja-fujita}). 
Catanese and Dettweiler provide examples where the monodromy image is not finite. 

Let us remark that after \cite{CD:Answer} the modified Xiao conjecture implies:
\begin{conj}
\label{conj:cong}
For any non-isotrivial fibred surface $f\colon S\to B$ of genus $g\geq 2$ such that the flat unitary summand has finite monodromy, 
it holds 
\begin{equation}\label{eq:cong}
u_f\leq \left\lceil \frac{g+1}{2} \right\rceil.
\end{equation}
\end{conj}
Indeed, given a non-isotrivial fibred surface $f\colon S\to B$ of genus $g\geq 2$ such that $\caU$ has finite monodromy, we can perform an (\'etale) base change $\pi\colon  B' \to B$  obtianing a new fibred surface $f'\colon S'\to B'$ such that  $\pi^*\caU$ is trivial, hence in particular $u_f=\rk\pi^* \mathcal{U}\leq q_{f'}$. So modified Xiao for $q_{f'}$ implies equation (\ref{eq:cong}).

\medskip

In the case of infinite monodromy the bound \ref{eq:cong} is false: this follows from Remark 38 of \cite{CD:Vector-bundles} and from the construction of Lu \cite{Lu:counterexample}. In the first version of the paper this bound was stated as a conjecture. See the Section \ref{sec:examples}.

\medskip

Hence, we see a serious discrepancy between the behavior of the relative irregularuity and the one of unitary rank. 
It is therefore extremely interesting to see that the bounds of (Q6) for $q_f$ can be extended to $u_f$. The main result of this paper answers positively to this question.

\begin{thm} \label{thm:main}
Let $f: S \ra B$ be a non-isotrivial fibration of genus $g$, flat unitary rank $u_f$ and Clifford index $c_f$. Then
\begin{equation}\label{eq:main}u_f \leq g - c_f.\end{equation}
Moreover, if the general fibre is a plane curve
of degree $d \geq 5$, then
\begin{equation}\label{eq:fnp} u_f \leq  g - c_f - 1.\end{equation}
\end{thm}

\begin{rmk}\label{rmk:sharpness}
The bound (\ref{eq:main}) is sharp in case of general (maximal) Clifford index $c_f= \left\lfloor\frac{g-1}{2}\right\rfloor$, as proven in the last Section \ref{sec:examples}. 
What about the non-general case? Modified Xiao's conjecture would imply that  the first bound of (Q.6) for the relative irregularity   is not the optimal one. For the unitary rank, it is an open and extremely interesting question to understand the case of non general Clifford index. If one believes in modified Xiao's conjecture (Conjecture \ref{conj:cong}), extremal examples have to be searched for in the cases where the unitary summand has inifinite monodromy. Unfortunately, the infinite series of examples constructed in \cite{CD:Vector-bundles} (she so-called Standard Case, which are the same used in \cite{Lu:counterexample}) all satisfy $c_f\leq 1$ (see Example \ref{ex:CD}).
 \end{rmk}

Let us spend a couple of words about the proof of Theorem \ref{thm:main}. 
The proof of the first inequality (\ref{eq:main}) follows the lines of the original argument of \cite{BGN}, but in this more general setting new results are needed.
The key points of our arguments are the deformation thecnhniques developed by the first author in \cite{Rigid} and by the third author and Pirola in \cite{PT},  together with an ad-hoc Castelnuovo-de Franchis theorem for tubular surfaces.

\medskip

As in \cite{BGN}, we start by constructing a {\em supporting divisor } for the family (see Definitions \ref{df:supp-infinit} and \ref{df:supp-family}), which is, rougly speaking, a divisor $\mathcal D$ on $S$ whose restriction on the general fibre supports the first-order deformation induced by $f$. 
This divisor  is constructed thanks to the results of the first named author in \cite{Rigid}.
The proof of \cite{BGN} then proceeds by treating separatedly the cases when $\mathcal D$ is relatively rigid or movable.

The main difference between the present case and \cite{BGN} stems from the fact that sections of $\caU$ do not correspond to global differential 1-forms on $S$, while sections of the trivial part do. 
Instead, after the recent work of Pirola and Torelli \cite{PT}, local {\em flat} sections of $\caU$ can be identified with {\em closed} holomorphic 1-forms on ``tubes'' $f^{-1}\left(\Delta\right) \subset S$ around smooth fibres (see Lemma \ref{lem:PT2}). This local liftability of sections of $\caU$ is enough to deal with the movable case.

To treat the rigid case, we develop a tubular version of the classical Castelnuovo-de Franchis theorem, that we believe to be interesting on its own. 
\begin{thm} \label{thm:tubularCdF}
Let $f: S \ra \Delta$ be a family of smooth compact/projective curves over a disk. Let $\omega_1,\ldots,\omega_k \in H^0\left(S,\Omega_S^1\right)$ ($k \geq 2$) be closed holomorphic 1-forms such that $\omega_i \wedge \omega_j = 0$ for every $i,j$, and whose restrictions to a general fibre $F$ are linearly independent. Then there exist a projective curve $C$ and a morphism $\phi: S \ra C$ such that $\omega_i \in \phi^*H^0\left(C,\omega_C\right)$ for every $i$ (possibly after shrinking $\Delta$).
\end{thm}
This Theorem can be applied to conclude our argument, precisely because of the abovementioned closedness of 1-forms corresponding to sections of $\caU$.

The inequality (\ref{eq:fnp}) for fibrations whose general fibres are plane curves uses  the results of \cite{FNP}: as these regard infinitesimal deformations, the very same arguments as above apply to extend Favale-Naranjo-Pirola's inequality to hold for $u_f$.

In the last Section we discuss the known examles with large unitary summand.

\medskip

We end the introduction by giving an application of our result to the framework of the Coleman-Oort conjecture on Shimura varieties in the Torelli locus. 
This is a straightforward application of the theory developed in \cite{LZ-hyper} and  \cite{LZ}, just plugging in inequality (\ref{eq:main}), as in the proof of \cite[Theorem 1.2.2]{LZ}. We refer to this article for the definitions and notations. Here we just recall the following facts.

Let $\mathcal M_g$ be the moduli space of smooth projective curves of genus $g\geq 2$ and $\mathcal A_g$ the moduli space of $g$-dimensional principally polarized abelian varieties; consider both spaces  with a full level $\ell$ structure, in order to be fine moduli spaces. 
Let $j^\circ \colon \mathcal M_g\rightarrow \mathcal A_g$ be the Torelli morphism. The Torelli locus $\mathcal T_g$ is the closure in $\mathcal A_g$ of the image of $j^\circ$.
Let $\mathcal M_g(c)$ be the locus of curves in $\mathcal M_g$ of Clifford index $ c$. 
Call $\mathcal T^\circ_{g}(c)$ the image of  $\mathcal M_g(c)$ in $\mathcal A_g$ via $j^\circ$. 
Let $\mathcal T_g(c)$ be the closure of $\mathcal T^\circ_g(c)$. Recall that for maximal Clifford index $c=\left\lfloor \frac{g-1}{2}\right\rfloor$, $\mathcal T_{g}(c)=\mathcal T_g$.
We say that a subvariety $Z\subset \mathcal A_g$ is contained generically in  $\mathcal T_g(c)$ if $Z\subseteq  \mathcal T_g(c)$ and $Z\cap  \mathcal T^\circ_g(c)\not =\emptyset$.

Recall that, given $C\subset \mathcal A_g$ a smooth closed curve, the canonical Higgs bundle $\mathcal E_C$ on $C$ is the Hodge bundle given by the universal family of abelian varieties restricted over $C$. The bundle $\mathcal E_C$ decomposes as $\mathcal E_C=\mathcal F_C\oplus \mathcal U_C$ where $\mathcal U_C$ is the maximal unitary Higgs subbundle corresponding to the maximal subrepresentation on which $\pi_1(C)$ acts through a compact unitary group. The bundles $\mathcal E_C$, $\mathcal F_C$ and $\mathcal U_C$ have an Hodge decomposition into a $(-1,0)$ and $(0,-1)$ part. 
\begin{cor}
Let $c$ be an integer between $1$ and $ \left\lfloor \frac{g-1}{2} \right\rfloor$.
Let $C\subset \mathcal A_g$ be a curve with Higgs 
bundle decomposition $\mathcal E_C=\mathcal F_C\oplus \mathcal U_C$. If 
$$\rk \mathcal U_C^{-1,0}>g-c$$
then $C$ is not contained generically in $\mathcal T_g(c)$. 
\end{cor}

\begin{rmk}
Hence in particular if 
$\rk \mathcal U_C^{-1,0}> \left\lceil \frac{g+1}{2} \right\rceil$
then, if $C$ is contained generically in the Torelli locus $\mathcal T_g$, the corresponding family of curves has not maximal Clifford index, which corresponds to not having general gonality.
\end{rmk}
\smallskip

\textbf{Acknowledgements:} The germ of this collaboration begun during the interesting Workshop  ``Birational Geometry of Surfaces'' at the University of Roma Tor Vergata on January 2016. We would like to express our gratitude toward Pietro Pirola for giving us the starting kick, and for stimulating and fruitful discussions on the topic. We thank Fabrizio Catanese for his kind interest and for giving us extremely useful suggestions.
We are also grateful to Xin Lu for pointing out a gap and a wrong conjecture in a previous version of this work. 
The first and second named authors whish to thank the Department of Mathematics of Pavia for the invitation and warm hospitality in February 2016. 

\medskip

\textbf{Basic assumptions and notation:} Throughout the whole article, all varieties are assumed to be smooth and defined over $\mC$. Unless otherwise is explicitly said, $f: S \ra B$ will be a fibration (a surjective morphism with connected fibres) from a compact surface $S$ to a compact curve $B$. The {\em genus} of $f$, defined as the genus of any smooth fibre, will be denoted by $g$, and assumed to be at least 2. 


\section{Some results on the Hodge bundle}

\subsection{Fujita decompositions of the Hodge bundle}

Although we are primarily interested in the case of fibred surfaces, in this section we can assume that $f: X \ra B$ is a fibration over a (smooth compact) curve $B$ of a compact K\"ahler manifold $X$ of arbitrary dimension $n+1 \geq 2$. We denote by $\omega_f = \omega_X \otimes \left(f^*\omega_B\right)^{\vee}$ the relative dualizing sheaf, and by $E_f = f_*\omega_f$ the {\em Hodge bundle} of $f$, which is a vector bundle on $B$ of rank equal to the geometric genus of a general fibre of $f$. Fujita investigated the structure of $E_f$ and obtained the following results. The complete proof of the second result is done by Catanese and Dettweiler in \cite{CD:Answer}.

\begin{thm}[Fujita decompositions \cite{Fuj1,Fuj2,CD:Answer}] \label{thm:Fuj-decomp}
Let $f: X \ra B$ be a fibration as above and $E_f = f_*\omega_f$ be its Hodge bundle. Then $E_f$ admits the following decompositions as sums of vector bundles
\begin{equation} \label{eq:Fuj-decomp}
E_f = \caO_B^{\oplus h} \oplus \caF = \caU \oplus \caA,
\end{equation}
where $h=h^1\left(B,f_*\omega_X\right)$, $\caF$ is nef with $h^1\left(B,\caF\otimes\omega_B\right)=0$, $\caU$ is flat and unitary, and $\caA$ is ample.
\end{thm}
The decompositions (\ref{eq:Fuj-decomp}) are the {\em first and second Fujita decompositions} of $E_f$. They are compatible, in the sense that $\caO_B^{\oplus h }\subseteq \caU$ and $\caA \subseteq \caF$. Moreover they can be combined into a finer decomposition
\begin{equation} \label{eq:combined-Fujita}
E_f = \caO_B^{\oplus h} \oplus \caU' \oplus \caA,
\end{equation}
where $\caU'$ is flat unitary and has no sections, $\caU = \caO_B^{\oplus h} \oplus \caU'$ and $\caF = \caU' \oplus \caA$. In the case $X = S$ is a surface, $h=q_f=q\left(S\right)-g\left(B\right)$ is the {\em relative irregularity} of $f$.

\begin{df} \label{df:flat-unitary-local-system}
The {\em flat unitary summand} (of $f$) is the vector bundle $\caU$ in the second Fujita decomposition. We denote by $\mU$ the associated local system, and by $\nabla=\nabla_{\caU}: \caU \ra \caU \otimes \omega_B$ the corresponding connection. The {\em flat unitary rank} of $f$ is $u_f := \rk \caU$.
\end{df}

\begin{rmk}\label{rmk: base change}
It is worth observing that the unitary rank $u_f$ is not necessarily constant  under finite base change; however it is non decreasing. That is, if $\pi: B' \ra B$ is a finite map, $X'$ is a desingularization of the fibre product $X \times_B B'$ and $f': X' \ra B'$ is the induced fibration, then we have that $u_{f'} \geq u_f$. 
A similar phenomenon happens also for the relative irregularity. 
Indeed, 
the Hodge bundle itself is not stable under base change, but becomes ``smaller'': $E_{f'}$ is a subsheaf of $\pi^*E_f$ and the quotient is a skyscraper sheaf supported on the $\pi$-branch points $b \in B$ such that $f^{-1}\left(b\right)$ is not semistable. 
Call $\caA$ (resp.  $\caA'$) and $\caU$ (resp.  $\caU'$) the ample and unitary summands of $E_f$ (resp. $E_{f'}$); observe that the composite morphism $\caU'\subseteq \pi^*\caU\oplus \pi^*\caA\ra \pi^*\caU $ is necessarily surjective (because the quotient $\pi^*E_f / E_{f'}$ is skyscraper), hence $u_{f'}=\rk \caU' \geq \rk\pi^*\caU =u_f$, as claimed.
This fact makes it possible that somehow, using a suitable base change $\pi$, part of the ample summand of $E_f$ becomes flat in $E_{f'}$, hence $u_{f'} > u_f$. 
\end{rmk}

We now sketch the construction of the unitary summand $\caU$. For more details, please consult \cite{CD:Answer}. Let $B^o \subseteq B$ denote the Zariski-open subset of non-critical values, and $f^o: f^{-1}\left(B^o\right) \ra B^o$ the restriction of $f$ to the smooth fibres. Over $B^o$, $E_f$ can be interpreted as part of a variation of Hodge structures
$$H^0\left(f^{-1}\left(b\right),\Omega^n_{f^{-1}\left(b\right)}\right) \subset H^n\left(f^{-1}\left(b\right),\mC\right).$$
Indeed, it is a holomorphic vector subbundle of $\caH = R^nf^o_*\mC \otimes \caO_{B^o}$, which is equipped with the Gau\ss-Manin connection $\nabla_{\caH}$ associated to the local system $\mH=R^nf^o_*\mC$. The basic idea is to define $\caU_{|B^o}$ as the vector subbundle of $E_f$ spanned by $\nabla_{\caH}$-flat sections, which is  flat and unitary with respect to the restriction of $\nabla_{\caH}$. 
This construction can be extended to the whole $B$ only if the monodromy of $\mH$ around the singular fibres is unipotent. 
This holds if and only if the fibres are semistable (recalling that we are assuming  the fibration to be relatively minimal with smooth total space).
Hence, in the semistable case, the monodromy of $\caU$ around the singular fibres becomes trivial (being unipotent and unitary). In case there is non-unipotent monodromy around a singular (non-semistable) fibre, $\caU_{|B^o}$ can be a strictly smaller subbundle, being the subbundle spanned by $\nabla_{\caH}$-flat sections that are also invariant by the local monodromy around the singular fibres.
From this perspecive we also can see the fact that the unitary rank is non-decreasing under finite base change, as observed in Remark \ref{rmk: base change}; moreover we can observe that in particular, choosing a $\pi$ such that $f'$ is semistable we maximize the unitary rank.

\subsection{Flat sections of $f_*\omega_f$ and 1-forms on $S$}

We will now more explicitly recall the interplay between the Hodge bundle and extensions of differential forms on the fibres in the case of a family of curves $f: S \ra B$, i.e. a fibration 
of relative dimension $n=1$ over a possibly non-compact base curve $B$. In order to simplify some statements we will assume that the family is semistable, which is enough to our purposes (see the end the previous section). Let $g$ be the (arithmetic) genus of the fibres of $f$, and for any $b \in B$ we denote by $C_b=f^{-1}\left(b\right)$ the corresponding fibre (with the induced subscheme structure).

\medskip
Consider now the sheaf of relative differentials $\Omega_{S/B}^1$ defined by the short exact sequence
\beq \label{eq:rel-diff}
0 \longra f^*\omega_B \longra \Omega_S^1 \longra \Omega_{S/B}^1 \longra 0.
\enq
The restriction of (\ref{eq:rel-diff}) to any nodal $C_b$ gives  the Kodaira-Spencer class $\xi_b \in \Ext^1\left(\Omega_{C_b}^1,\caO_{C_b}\right) \otimes T_{B,b}^{\vee}$ associated to the first-order infinitesimal deformation of $C_b$ induced by $f$.

There is a natural map $\Omega_{S/B}^1 \rightarrow \omega_f$, which is an isomorphism outside the singular fibres. Moreover, it is an injection if and only if every fibre of $f$ is reduced, and in particular if $f$ is semistable. In this case, $\Omega_{S/B}^1$ is torsion-free, and therefore also $f_*\Omega_{S/B}^1$ is torsion-free, hence a vector bundle on $B$. Furthermore, $f_*\Omega_{S/B}^1$ is a subsheaf of the Hodge bundle $E_f=f_*\omega_f$, and there is a connecting homomorphism
$$f_*\Omega_{S/B}^1 \stackrel{\partial}{\longra} R^1f_*\caO_S \otimes \omega_B$$
which at a regular value $b \in B$ is nothing but the cup product with the corresponding Kodaira-Spencer class $\cup\xi_b: H^0\left(C_b,\omega_{C_b}\right) \ra H^1\left(C_b,\caO_{C_b}\right) \otimes T_{B,b}^{\vee}.$ Denote by $\caK$ the kernel of $\partial$, which is a vector bundle on $B$ sitting into the short exact sequence
\begin{equation} \label{eq:sesK}
0 \longra \omega_B \longra f_*\Omega_S^1 \longra \caK \longra 0.
\end{equation}
Sections of $\caK$ correspond thus to families of forms on the fibres that locally glue to give a form on a ``tubular'' open neighbourhood of any fibre. Note that the pull-back of a form on $B$ is trivial along the fibres, which also explains the exact sequence (\ref{eq:sesK}).

Suppose now that $S$ and $B$ are compact. Since over the regular values $B^o$, the Hodge bundle $E_f$ coincides with $f_*\Omega_{S/B}^1$, it makes sense to compare $\caK$ with the Fujita decompositions of $E_f$. Indeed, the trivial summand $\caO_B^{\oplus q_f}$ of $E_f$ corresponds to the holomorphic forms on the fibres that extend to the whole surface, that is, to $H^0\left(S,\Omega_S^1\right) / f^*H^0\left(B,\Omega_B^1\right)$. As for the flat unitary summand $\caU$, we have the following result of Pirola and the third named author.


\begin{lem}[Lifting Lemma in \cite{PT}] \label{lem:PT2}
Let $f\colon S \ra B$ be a semistable fibration, $\mU$ the local system associated to the flat unitary summand, and $\Omega_{S,d}^1$ the sheaf of closed holomorphic 1-forms. Then there is a short exact sequence of sheaves (of abelian groups) on $B$
\beq \label{eq:PT-sheaves}
0 \longra \omega_B \longra f_*\Omega_{S,d}^1 \longra \mU \longra 0.
\enq
In particular, if $V \subset B$ is a disk, then there is an exact sequence of vector spaces:
\beq \label{eq:PT-sections}
0 \longra H^0\left(V,\omega_B\right) \longra H^0\left(f^{-1}\left(V\right),\Omega_{S,d}^1\right) \longra H^0\left(V,\mU\right) \longra 0.
\enq
\end{lem}
This means that the {\em flat} sections of $\caU$ correspond to 1-forms on the fibres that extend to {\em closed} 1-forms on tubular neighbourhoods, and thus arbitrary sections of $\caU$ over $V \subseteq B$ correspond to $\caO_B\left(V\right)$-linear combinations of closed 1-forms on $f^{-1}\left(V\right)$. Note that here there is a crucial difference between the local system $\mU$ and the vector bundle $\caU$. 

Then one gets immediately the following:
\begin{lem}[\cite{PT}]\label{lem: sequence}
Under the above assumptions, there is an injection of sheaves $i\colon \mathbb U \hookrightarrow \mathcal K$, that sits in the following commutative diagram:
$$
\xymatrix{
0 \ar[r]&\omega_B\ar[r]\ar[d]& f_*\Omega_{S,d}^1\ar[d]\ar[r]& \mU\ar[d]_i\ar[r] &0\\
0 \ar[r]&\omega_B\ar[r] &f_*\Omega_{S}^1\ar[r]& \mathcal K\ar[r] &0
}
$$
\end{lem}
\begin{rmk}
The inclusion of Lemma \ref{lem: sequence} is actually a quite natural fact. 
Indeed, 
over the regular vaules $B^o$, the connecting homomorphism $\partial$ coincides with the composition
$$E_{f^o}=f^o_*\omega_{f^o} \stackrel{\nabla_{\caH}}{\longra} \caH\otimes\omega_{B^o} \longra \left(\caH/E_{f^o}\right) \otimes \omega_{B^o} \cong R^1f^o_*\caO_{S^o} \otimes \omega_{B^o}$$
of the Gau\ss-Manin connection (restricted to $E_{f^o}$) and the natural projection. Thus the kernel of the first map (which spans $\caU$) has to be contained in the kernel $\caK$ of $\partial$.
\end{rmk}



\subsection{Supporting divisors}

We recall now some definitions and results concerning infinitesimal deformations of curves already introduced in \cite{ColPir, Rigid, BGN}. 
Let $C$ be a smooth curve and $\xi \in H^1\left(C,T_C\right)$ the Kodaira-Spencer class of a first-order infinitesimal deformation $C \hra \caC$,
which corresponds to the exact sequence of vector bundles on $C$
$$0 \ra N_{C/\caC}^{\vee} \cong \caO_C \ra \Omega_{\caC|C}^1 \ra \omega_C \ra 0.$$

\begin{df} \label{df:supp-infinit}
We say that $\xi$ is {\em supported on} an effective divisor $D$ if and only if
\beqn
\xi \in \ker\left(H^1\left(C,T_C\right) \longra H^1\left(C,T_C\left(D\right)\right)\right).
\enqn
Furthermore, $\xi$ is {\em minimally} supported on $D$ if it is not supported on any strictly effective subdivisor $D' < D$.
\end{df}

Equivalently, $\xi$ is supported on $D$ if and only if the subsheaf $\omega_C\left(-D\right)$ lifts to $\Omega_{\caC|C}^1$, i.e. the inclusion $\omega_C\left(-D\right) \hra \omega_C$ factors through $\Omega_{\caC|C}^1 \ra \omega_C$.

Recall that the connecting homomorphism in the long exact sequence of cohomology associated to $\xi$ is given by cup-product with $\xi$
$$\partial_{\xi}: H^0\left(C,\omega_C\right) \stackrel{\cup\xi}{\longra} H^1\left(C,\caO_C\right).$$

\begin{df} \label{df-rk}
The {\em rank} of $\xi$ is defined as $$\rk \xi = \rk \partial_{\xi} = \rk\left(\cup\xi\right).$$
\end{df}

The rank of $\xi$ is the most important numerical invariant for our purposes, and is related to supporting divisors by the following

\begin{thm}[Lemma 2.3 and Thm 2.4 in \cite{BGN}] \label{thm-supp-bound}
Suppose $\xi$ is supported on $D$. Then $H^0\left(C,\omega_C\left(-D\right)\right) \subseteq \ker \partial_{\xi}$. In particular,
\beqn
\rk \xi \leq \deg D - r\left(D\right).
\enqn
If moreover $\xi$ is minimally supported on $D$, then
\beqn
\rk \xi \geq \deg D - 2r\left(D\right) = \Cliff\left(D\right).
\enqn
\end{thm}

Note that $H^1\left(C,T_C\left(D\right)\right)=0$ if $D$ is ample enough (e.g. $\deg D > 2g\left(C\right)-2$), hence such $D$ supports any deformation $\xi$. However, these divisors are uninteresting, since for example the inequalities provided by Theorem \ref{thm-supp-bound} are trivial. We are therefore interested in the existence of small supporting divisors, which can be constructed using the adjoint theorem (see \cite{ColPir}).
In \cite{Rigid} the following numerical sufficient condition for the existence of supporting divisors is obtained.

\begin{pro}[\cite{Rigid} Corollary 3.1] \label{pro:num-cond-inf}
If $\dim\ker\left(\partial_{\xi}\right) > \frac{g\left(C\right)+1}{2}$, then there is a two-dimensional subspace $W \subseteq \ker\left(\partial_{\xi}\right) \subset H^0\left(C,\omega_C\right)$ whose base divisor $D$ supports $\xi$.
\end{pro}

We extend now the previous definitions to the case of one-dimensional families of curves as done in \cite{Rigid}. In order to simplify the exposition, we will assume that the family has no singular members, though the theroy can be developed in more general cases.

Let $f:S\ra B$ be a family of smooth curves, i.e. a {\em smooth} proper surjective morphism with connected fibres from a (non-necessarily compact) surface $S$ to a (non-necessarily compact) curve $B$. Denote by $C_b = f^{-1}\left(b\right)$ the fibre over a point $b \in B$, which is smooth of genus $g$, and by $\xi_b \in H^1\left(C_b,T_{C_b}\right)$ the Kodaira-Spencer class of the first-order deformation of $C_b$ induced by $f$ (well-defined up to non-zero scalar multiplication).

\begin{df} \label{df:supp-family}
We say that an effective divisor $\caD\subset S$ supports $f$ if the restrictions $D_b=\caD_{|C_b}$ support $\xi_b$ for general $b \in B$.
\end{df}

Note that Definition \ref{df:supp-family} only considers general fibres, hence we can always assume that a divisor $\caD$ supporting a family $f$ contains no fibre.

The following result is a simplification of some results in \cite[Sec. 2]{Rigid}, which is enough for the setting of this work.

\begin{lem} \label{lem:local-splitting}
Suppose that $B$ is a disk, $f: S \ra B$ has no singular fibres, and $\caD \subset S$ is a divisor supporting $f$ such that $\caD \cdot F < 2g-2$ for any fibre $F$. Then the inclusion $\omega_f\left(-\caD\right) \hra \omega_f$ factors uniquely as
$$\omega_f\left(-\caD\right) \stackrel{\iota}{\hra} \Omega_S^1 \twoheadrightarrow \omega_f.$$
\end{lem}
\begin{proof}

Note first that $\deg\left(\omega_f\left(-\caD\right)_{|F}\right)=2g-2-\caD \cdot F > 0$ for any fibre $F$. Hence $f_*\left(\omega_f\left(-\caD\right)^{\vee}\right)=0$ and
$$\Hom\left(\omega_f\left(-\caD\right),f^*\omega_B\right)=H^0\left(B,f_*\left(\omega_f\left(-\caD\right)\right)^{\vee} \otimes \omega_B\right)=0.$$
The (left-exact) functor $\Hom\left(\omega_f\left(-\caD\right),-\right)$ applied to the exact sequence
$$\xi: \quad 0 \longra f^*\omega_B \longra \Omega_S^1 \longra \omega_f \longra 0$$
gives
$$0 \longra \Hom\left(\omega_f\left(-\caD\right),\Omega_S^1\right) \longra \Hom\left(\omega_f\left(-\caD\right),\omega_f\right) \stackrel{\mu}{\longra} \Ext^1\left(\omega_f\left(-\caD\right),f^*\omega_B\right).$$
The image of a morphism of sheaves $\phi: \omega_f\left(-\caD\right) \ra \omega_f$ by the map $\mu$ is $\phi^* \xi$, the class of the short exact sequence obtained from $\xi$ by pull-back in the last term and completing the diagram. Thus the natural inclusion $\iota$ factors through $\Omega_S^1$ if and only if $\mu\left(\iota\right)=0$.

By \cite[Lemma 2.4]{Rigid}, $\Ext^1\left(\omega_f\left(-\caD\right),f^*\omega_B\right) \cong H^0\left(B,\caE\right)$, where $\caE = \caExt_f^1\left(\omega_f\left(-\caD\right),\caO_S\right)\otimes\omega_B$ is a vector bundle whose fibre over a point $b \in B$ is
$$\Ext^1\left(\omega_{F_b}\left(-\caD_{|F_b}\right),T_{B,b}^{\vee}\right)\cong H^1\left(F_b,T_{F_b}\left(\caD_{|F_b}\right)\right)\otimes T_{B,b}^{\vee}$$
(cf. \cite[Lemma 2.3, Proposition 2.1 and Appendix]{Rigid}). Moreover, under this isomorphism $\mu\left(\iota\right)$ corresponds to a section $\sigma$ of $\caE$, whose value $\sigma\left(b\right)$ at any $b\in B$ is the class of the restriction to $F_b$ of the pull-back $\iota^*\xi$. Since this pull-back and the restriction clearly conmute, $\sigma\left(b\right)$ coincides with the image of $\xi_b \in H^1\left(F_b,T_{F_b}\right)$ in $H^1\left(F_b,T_{F_b}\left(\caD_{|F_b}\right)\right)$, which is generically zero (hence identically zero, $\caE$ being torsion-free) because the family $f$ is supported on $\caD$. This shows that $\mu\left(\iota\right)=0$ and completes the proof.
\end{proof}

Note that, in particular, if $\caD$ supports $f$, then we have the following injection
$$f_*\omega_f\left(-\caD\right) \subseteq \caK = \ker\left(f_* \omega_f \ra R^1f_*\caO_S \otimes \omega_B\right).$$

\begin{pro} \label{prop:local-supp-div} Suppose $\rk \caK > \frac{g+1}{2}$. Then there is an open disk $V \subseteq B$ and a divisor $\caD \subset f^{-1}\left(V\right)$ {\em minimally} supporting the restriction of $f$ and such that $h^0\left(C_b,\omega_{C_b}\left(-\caD_{|C_b}\right)\right) \geq 2$ and $H^0\left(C_b,\omega_{C_b}\left(-\caD_{|C_b}\right)\right) \subseteq \caK_b$ for any fibre $C_b$ with $b \in V$.
\end{pro}
\begin{proof}
The proof is analogous to that of Theorem 3.3 and Corollary 3.2 in \cite{Rigid}. Roughly speaking, applying Proposition \ref{pro:num-cond-inf} pointwise we can locally (over any disk $V \subseteq B$) find a rank-2 vector subbundle $\caW \subseteq \caK$ such that the divisorial base locus $\caD$ of the relative evaluation map
$$f^*\caW \longra f^*f_*\omega_f \longra \omega_f$$
supports $f$. Up to shrinking $V$ we may assume that $\caD$ consists of disjoint sections of $f$ over $V$. If $\caD$ is not minimal supporting $f$, we can remove some of the components (or reduce their multiplicities) until obtaining a minimal one.
\end{proof}


\subsection{ A Castelnuovo-de Franchis theorem for tubular surfaces}

We now give the proof of Theorem \ref{thm:tubularCdF}. Our argument follows the one of Beauville \cite[Proposition X.6]{Beau}.
Since all the 1-forms $\omega_1,\ldots,\omega_k$ are pointwise proportional, there are meromorphic functions $g_2,\ldots,g_k$ on $S$ such that $\omega_i = g_i\omega_1$. Differentiating these equalities and using that $d\omega_i=0$, we obtain $0 = dg_i \wedge \omega_1$ for $i = 2,\ldots,k$. In particular, there is another meromorphic function $g_1$ such that $\omega_1 = g_1 dg_2$, and also $0 = dg_1 \wedge dg_2$. This means that the meromorphic differentials $dg_2, \ldots, dg_k$ are pointwise proportional to $\omega_1$ (wherever they are defined) and also to $dg_1$, hence $dg_i \wedge dg_j = 0$ for any $i,j$. 

These functions define a meromorphic map $\psi: S \dra \mP^k$ as $$\phi(p) = \left(1:g_1(p):g_2(p):\ldots:g_k(p)\right).$$ Let $\epsilon: \hat{S} \ra S$ be a resolution of the indeterminacy locus of $\psi$, and let $\hat{\psi}: \hat{S} \ra \mP^k$ be the corresponding holomorphic map. Note that this resolution might not exist on the original surface due to the possible existence of infinitely many indeterminacy points. However, over a smaller disk there are only finitely many such points, and each of them can be resolved after finitely many blow-ups. If $\left(x_1,\ldots,x_k\right)$ are affine coordinates on $\left\{X_0\neq0\right\} \subset \mP^k$, then by construction $\epsilon^*\omega_1 = \hat{\psi}^*\left(x_1dx_2\right)$ and $\epsilon^*\omega_i = \hat{\psi}^*\left(x_1x_idx_2\right)$ for any $i \geq 2$. Furthermore $\hat{\psi}^*\left(dx_i \wedge dx_j\right) = dg_i \wedge dg_j = 0$, which implies that the image of $\hat{\psi}$ is locally an analytic curve $\hat{C} \subset \mP^k$. Since $\hat{S}$ is smooth, $\hat{\psi}$ factors throught the normalization $\nu: C \ra \hat{C}$, giving a holomorphic map $\hat{\phi}:\hat{S} \ra C$.

The meromorphic 1-forms on $C$ defined as $\alpha_1 = \nu^*\left(x_1dx_2\right)$ and $\alpha_i = \nu^*\left(x_1x_idx_2\right)$ (for $i\geq 2$) verify that $\hat{\phi}^*\alpha_i = \epsilon^*\omega_i$ are holomorphic. A straightforward computation in local coordinates shows that 
$$\Div\left(\hat{\phi}^*\alpha_i\right) = \hat{\phi}^*\Div\left(\alpha_i\right) + \sum \left(n_j-1\right) E_j,$$
where the $E_j$ are the irreducible components of the fibres of $\hat{\phi}$, and $n_j$ is the corresponding multiplicity. Since the $\hat{\phi}^*\alpha_i$ are holomorphic, these divisors are effective, but any contribution of a pole of $\alpha_i$ to $E_j$ would be smaller or equal than $-n_j$. Therefore the $\alpha_i$ are holomorphic 1-forms on $C$.

To conclude, let $F \subset \hat{S}$ be a general fibre of $\epsilon \circ f$, and let $\pi: F \ra C$ be the map induced by $\hat{\phi}$. Since the $\pi^*\alpha_i = \left(\epsilon^*\omega_i\right)_{|F}$ are linearly independent by hypothesis, in particular they are not zero and $\pi$ is hence surjective. This implies that $C$ is indeed a compact curve of genus $g\left(C\right) \geq k$.

Note that, {\em a fortiori}, the resolution of indeterminacy $\epsilon$ is not necessary, since every $\epsilon$-exceptional divisor is contracted by $\hat{\phi}$ because $g\left(C\right) \geq  k \geq 2$.


\section{Bounding $u_f$: proof of Theorem \ref{thm:main}}

Recall that the {\em Clifford index} of a smooth projective curve $C$ is defined as
$$\Cliff\left(C\right) = \min \left\{\Cliff\left(D\right)=\deg D - 2r\left(D\right) \,|\, h^0\left(C,\caO_C\left(D\right)\right),h^1\left(C,\caO_C\left(D\right)\right) \geq 2\right\}.$$
It satisfies $0 \leq \Cliff\left(C\right) \leq \left\lfloor\frac{g-1}{2}\right\rfloor$, and the second is an equality for general $C \in \caM_g$. Given a fibration $f:S \ra B$, the {\em Clifford index} of $f$ is $c_f$, the maximal Clifford index of the smooth fibres, which is attained over a Zariski-open subset of $B$.

First of all, since the flat unitary rank $u_f$ does not decrease by finite base change, we can assume that the fibration is semistable by applying the semistable reduction. Consider now an open disk $V \subseteq B$ such that:
\begin{itemize}
\item $V$ contains only regular values and the corresponding smooth fibres have Clifford index $c_f$, and
\item the connecting homomorphism $\partial: f_*\Omega_{S/B}^1 \ra R^1f_*\caO_S \otimes \omega_B$ has constant rank on $V$.
\end{itemize}
From now on, let $f: S \ra B$ denote only the restriction to $f^{-1}\left(V\right) \ra V$. Thus $f$ is a smooth fibration, $\Omega_{S/B}^1 \cong \omega_f$, and $\caK = \ker \partial$ is a vector bundle whose fibre $K_b$ over any $b \in B$ is exactly $\ker\left(\cup\xi_b: H^0\left(C_b,\omega_{C_b}\right) \ra H^1\left(C_b,\caO_{C_b}\right)\right)$.

\subsection{Proof of inequality  (\ref{eq:main})} 
Suppose by contradiction that $u_f > g-c_f$. Since $c_f \leq \left\lfloor\frac{g-1}{2}\right\rfloor$, in particular $u_f > \frac{g+1}{2}$. Therefore, by Proposition \ref{prop:local-supp-div} there is (up to shrinking $B$) a divisor $\caD \subset S$ {\em minimally} supporting $f$ such that
$$h^1\left(C_b,\caO_{C_b}\left(\caD_{|C_b}\right)\right)=h^0\left(C_b,\omega_{C_b}\left(-\caD_{|C_b}\right)\right) \geq 2$$
for general $b \in B$. 

As in the proof of \cite[Theorem 1.2]{BGN}, we consider now two cases:
\begin{enumerate}
\item[Case 1:] The divisor $\caD$ is relatively rigid, that is $h^0\left(C_b,\caO_{C_b}\left(\caD_{|C_b}\right)\right) = 1$ for a general $b \in B$. Then Theorem \ref{thm-supp-bound}, together with Riemann-Roch, implies that
$$H^0\left(C_b,\omega_{C_b}\left(-\caD_{|C_b}\right)\right) = K_b,$$
and hence $\caK=f_*\omega_f\left(-\caD\right)$.

Let now $\omega_f\left(-\caD\right) \hra \Omega_S^1$ be the map  provided by Lemma \ref{lem:local-splitting}, which gives a splitting $\caK \hra f_*\Omega_S^1$. Thus every section of $\caK$ corresponds to a 1-form on $S$, and in particular every $w \in \ker \cup\xi_b \subseteq H^0\left(C_b,\omega_{C_b}\right)$ can be extended to $S$. Moreover, all these extensions are sections of the same line bundle $\omega_f\left(-\caD\right)$, and therefore any two such extensions wedge to 0. 

Let now $\eta_1,\ldots,\eta_{u_f}$ be a basis of flat sections of $\caU \subseteq \caK$, 
i.e. a basis of $H^0\left(V,\mU\right)$, and let $\omega_1,\ldots,\omega_{u_f} \in H^0\left(S,\Omega_S^1\right)$ be the extensions provided by the splitting  above. By the previous discussion, any two of these forms  wedge to zero. We are thus in the situation of Theorem \ref{thm:tubularCdF}, which gives a new fibration $\phi: S \ra C$ over a compact curve of genus $g\left(C\right) = u_f$. Let now $C_b$ be any fibre of $f$, and $\pi: C_b \ra C$ be the restriction of $\phi$ to the smooth fibre $C_b$. Applying Riemann-Hurwitz we obtain
\beqn
2g-2 \geq \deg\pi \, (2u_f-2).
\enqn
At the beginning of the proof we obtained that $u_f > \frac{g+1}{2}$, so that $2u_f-2 > g-1$, and thus
\beqn
2(g-1) > \deg\pi \, (g-1).
\enqn
It follows that $\deg\pi=1$, so every smooth fibre is isomorphic to $C$ and hence $f$ is isotrivial.

\item[Case 2:] The divisor $\caD$ moves on any smooth fibre, i.e. $h^0\left(C_b,\caO_{C_b}\left(\caD_{|C_b}\right)\right) \geq 2$ for every regular value $b \in B$. We use Theorem \ref{thm-supp-bound} to obtain
\beq \label{ineq-final-proof}
\rk \xi_b \geq \Cliff\left(\caD_{|C_b}\right) \geq c_f.
\enq
But $\caU_b \subseteq \ker \partial_{\xi_b} = K_{\xi_b}$, so that $\rk \xi_b = g - \dim K_{\xi_b} \leq g - u_f$, and the inequality (\ref{ineq-final-proof}) implies that
\beqn
g - u_f \geq c_f,
\enqn
contradicting our very first hypothesis.
\end{enumerate}
\subsection{Proof  of inequality (\ref{eq:fnp})}
Let us now assume that the general fibre of $f$ is a plane curve of degree $d \geq 5$, and note that the Clifford index is thus $c_f=d-4$. 
If $C$ is a general fibre and $0\neq\xi \in H^1\left(C,T_C\right)$ is the induced infinitesimal deformation, \cite{FNP}[Theorem 1.3] gives directly
$\rk \xi \geq d-3= c_f+1$, and hence $$\rk \mathcal{U} \leq \dim \ker \xi = g - \rk \xi \leq g-c_f-1.$$

\section{Examples}\label{sec:examples}
In this last section we analize some examples satisfying the equality $u_f=\left\lceil\frac{g+1}{2}\right\rceil$. 
The examples of Pirola and Albano-Pirola satisfy $u_f=\left\lceil \frac{g+1}{2}\right\rceil$, and moreover they are extremal with respect to (\ref{eq:BGN}), as we verify in Example \ref{ex:AP}.
The examples of Catanese and Dettweiler (in the standard case in the notations of \cite{CD:Vector-bundles}) satisfy the above equality, but are not extremal with respect to (\ref{eq:main}), as they are not of general gonality, as proved here: in fact the gonality is less than or equal to three. Note that these examples have $q_f=0$.

It remains open to find examples with both $c_f$ and $u_f$ big, but $q_f$ small. The first step is to investigate the non-standard examples of Catanese and Dettweiler \cite{CD:Vector-bundles}.


\begin{ex}[cf. \cite{alb-pir}] \label{ex:AP}
We first consider some non-isotrivial fibrations $f:S \rightarrow B$ constructed by Albano and Pirola in \cite{alb-pir} with invariants $(q_f,g)=(4,6), (6,10)$. 
We will now show that these fibrations satisfy $q_f=u_f$.
 
The general fibres $C_b$ of $f$ are simultaneously double covers of a fixed curve $D$ and \'etale cyclic covers (of prime order $p$) of hyperelliptic curves $E_b$ with simple Jacobian variety. Moreover, the Jacobian of $C_b$ is isogenous to $J\left(D\right) \times J\left(D\right) \times J\left(E_b\right)$, and the simplicity of $J\left(E_b\right)$ implies that the relative irregularity of $f$ is $q_f = 2 g\left(D\right)$ and does not grow under base change.

The curves $E_b$'s form an hyperelliptic fibration $h: T \rightarrow B$ and (possibly after finite base change) there is a compatible \'etale cyclic cover $S \rightarrow T$ of order $p$. Thus the Hodge bundle of $h$ is a direct summand of the Hodge bundle of $f$, and is complementary to the trivial summand $\mathcal{O}_B^{\oplus q_f} = \mathcal{O}_B^{\oplus 2g\left(D\right)}$ that corresponds to the constant factors $J\left(D\right) \times J\left(D\right)$ of the Jacobians (up to isogeny). Thus the non-trivial flat unitary summand $\mathcal{U}'$ of $E_f$ (see equation (\ref{eq:combined-Fujita})) coincides with the flat unitary summand of $E_{h}$.

Suppose that $u_f > q_f$, that is, $\rk \mathcal{U}' > 0$. Since $h$ is an hyperelliptic fibration, by \cite[Theorem Theorem 4.7]{LZ-hyper} the monodromy representation of $\pi_1\left(B\right)$ corresponding to $\mathcal{U}'$ has finite non-trivial image. But this would provide a finite \'etale covering $\pi: \widetilde{B} \rightarrow B$ such that $\pi^* \mathcal{U}'$ is a trivial summand of the Hodge bundle of base-change of $h$, contradicting the simplicity of the Jacobian of the general fibres of $h$.
\end{ex}

\begin{ex}[cf. \cite{CD:Semiample} and \cite{CD:Vector-bundles}] \label{ex:CD}
Catanese and Dettweiler construct infinite families of  non-isotrivial fibrations $f\colon S \rightarrow B$ with big values of $u_f$. 
We will now show that the Clifford index of these fibrations is $c_f=1 < \left\lfloor\frac{g-1}{2}\right\rfloor$, hence it is not extremal for Theorem \ref{thm:main}.

We only recall the (few) elements of the construction of the so-called standard case (\cite[Definition 4.1]{CD:Vector-bundles}) needed for our purpose. 
Let $n$ be any integer such that $GCD(n,6)=1$. The smooth fibres of $f$ are cyclic coverings $\pi: C \rightarrow \mathbb{P}^1$ of degree $n$ branched on 4 points. 
These curves can be realized as the normalizations of the plane curves with equation (we use the same notations as in \cite{CD:Vector-bundles}):
$$
C_x: z_1^n-y_0y_1(y_1-y_0)(y_1-x y_0)^{n-3}=0,
$$
with $x \in \mathbb{C}\setminus\{0,1\}$ and coordinates $(y_0:y_1:z_1)\in \mathbb P^2$.
Now, we can just observe that the point $(1:x:0)$ has multiplicity $n-3$ in the curve $C_x$, so by projecting from this point we get a $3$ to $1$ map to $\mathbb P^1$, hence a $g^1_3$ on the normalization of $C_x$.

\end{ex}

\bibliography{biblio}

\begin{thebibliography}{BGAN15}

\bibitem[AP16]{alb-pir}
Alberto Albano and Gian~Pietro Pirola.
\newblock Dihedral monodromy and {X}iao fibrations.
\newblock {\em Ann. Mat. Pura Appl. (4)}, 195(4):1255--1268, 2016.

\bibitem[Bar00]{barja-fujita}
Miguel~\`Angel Barja.
\newblock On a conjecture of {F}ujita.
\newblock {\em Available on the ResearchGate page of the author}, 2000.

\bibitem[Bea96]{Beau}
Arnaud Beauville.
\newblock {\em Complex algebraic surfaces}, volume~34 of {\em London
  Mathematical Society Student Texts}.
\newblock Cambridge University Press, Cambridge, second edition, 1996.
\newblock Translated from the 1978 French original by R. Barlow, with
  assistance from N. I. Shepherd-Barron and M. Reid.

\bibitem[BGAN15]{BGN}
Miguel~\'{A}ngel Barja, V\'{\i}ctor Gonz\'alez-Alonso, and Juan~Carlos Naranjo.
\newblock Xiao's conjecture for general fibred surfaces (arxiv:1401.7502).
\newblock {\em J. Reine Angew. Math.}, To appear, 2015.

\bibitem[Cai98]{Cai}
Jin-Xing Cai.
\newblock Irregularity of certain algebraic fiber spaces.
\newblock {\em Manuscripta Math.}, 95(3):273--287, 1998.

\bibitem[CD14]{CD:Semiample}
Fabrizio Catanese and Michael Dettweiler.
\newblock The direct image of the relative dualizing sheaf needs not be
  semiample.
\newblock {\em C. R. Math. Acad. Sci. Paris}, 352(3):241--244, 2014.

\bibitem[CD16]{CD:Vector-bundles}
Fabrizio Catanese and Michael Dettweiler.
\newblock Vector bundles on curves coming from variation of {H}odge structures.
\newblock {\em Internat. J. Math.}, 27(7):1640001, 25, 2016.

\bibitem[CD17]{CD:Answer}
Fabrizio Catanese and Michael Dettweiler.
\newblock Answer to a question by {F}ujita on {V}ariation of {H}odge
  {S}tructures.
\newblock {\em Adv. Stud. in Pure Math.}, 74-04(04):73--102, 2017.

\bibitem[CLZ16]{LZ}
Ke~Chen, Xin Lu, and Kang Zuo.
\newblock On the {O}ort conjecture for {S}himura varieties of unitary and
  orthogonal types.
\newblock {\em Compos. Math.}, 152(5):889--917, 2016.

\bibitem[CP95]{ColPir}
Alberto Collino and Gian~Pietro Pirola.
\newblock The {G}riffiths infinitesimal invariant for a curve in its
  {J}acobian.
\newblock {\em Duke Math. J.}, 78(1):59--88, 1995.

\bibitem[Deb82]{Deb-Beau}
Olivier Debarre.
\newblock In\'egalit\'es num\'eriques pour les surfaces de type g\'en\'eral.
\newblock {\em Bull. Soc. Math. France}, 110(3):319--346, 1982.
\newblock With an appendix by A. Beauville.

\bibitem[FPN17]{FNP}
Filippo Favale, Gian~Pietro Pirola, and Juan~Carlos Naranjo.
\newblock On the {X}iao conjecture for plane curves.
\newblock {\em Geometriae Dedicata (to appear)}, 2017.

\bibitem[Fuj78a]{Fuj1}
Takao Fujita.
\newblock On {K}\"ahler fiber spaces over curves.
\newblock {\em J. Math. Soc. Japan}, 30(4):779--794, 1978.

\bibitem[Fuj78b]{Fuj2}
Takao Fujita.
\newblock The sheaf of relative canonical forms of a {K}\"ahler fiber space
  over a curve.
\newblock {\em Proc. Japan Acad. Ser. A Math. Sci.}, 54(7):183--184, 1978.

\bibitem[GA16]{Rigid}
V{\'{\i}}ctor Gonz{\'a}lez-Alonso.
\newblock On deformations of curves supported on rigid divisors.
\newblock {\em Ann. Mat. Pura Appl. (4)}, 195(1):111--132, 2016.

\bibitem[Lu17]{Lu:counterexample}
Xin Lu.
\newblock Family of curves with large unitary factor in the {H}odge bundle,
  2017.

\bibitem[LZ14]{LZ-hyper}
Xin Lu and Kang Zuo.
\newblock The {O}ort conjecture on {S}himura curves in the {T}orelli locus of
  curves (arxiv:1405.4751), 2014.

\bibitem[Pir92]{Pir-Xiao}
Gian~Pietro Pirola.
\newblock On a conjecture of {X}iao.
\newblock {\em J. Reine Angew. Math.}, 431:75--89, 1992.

\bibitem[PT17]{PT}
Gian~Pietro Pirola and Sara Torelli.
\newblock Massey products and {F}ujita decompositions.
\newblock {\em Preprint arXiv:1710.02828 [math.AG]}, 2017.

\bibitem[Ser96]{Ser-Iso}
Fernando Serrano.
\newblock Isotrivial fibred surfaces.
\newblock {\em Ann. Mat. Pura Appl. (4)}, 171:63--81, 1996.

\bibitem[Xia87a]{Xiao-5/6}
Gang Xiao.
\newblock Fibered algebraic surfaces with low slope.
\newblock {\em Math. Ann.}, 276(3):449--466, 1987.

\bibitem[Xia87b]{Xiao-P1}
Gang Xiao.
\newblock Irregularity of surfaces with a linear pencil.
\newblock {\em Duke Math. J.}, 55(3):597--602, 1987.

\end{thebibliography}
\bibliographystyle{alpha}

\noindent{Gottfried Wilhelm Leibniz Universit\"at Hannover,
Institut f\"ur Algebraische Geometrie\\
Welfengarten 1, 30167 Hannover (Germany)}\\
\textsl{gonzalez@math.uni-hannover.de}

\medskip
\noindent{Dipartimento di Scienza e Alta Tecnologia, Universit\`a dell'Insubria,\\ Via Valleggio 11, 22100, Como, Italy}\\
\textsl{lidia.stoppino@uninsubria.it}
\medskip

\noindent{Dipartimento di Matematica,Universit\`a di Pavia,\\ Via Ferrata 1, 27100, Pavia, Italy}\\
\textsl{sara.torelli02@universitadipavia.it}

\end{document}